\documentclass{amsart}

\usepackage[utf8]{inputenc}
\usepackage{amsmath}
\usepackage{amsthm}
\usepackage{amssymb}
\usepackage{amsfonts}
\usepackage[shortlabels]{enumitem}
\usepackage[pdftex,colorlinks,citecolor=blue, linkcolor=blue]{hyperref}
\usepackage{centernot} 
\usepackage{tikz}
\usetikzlibrary{arrows}
\usetikzlibrary{shapes,decorations}
\usetikzlibrary{positioning}
\usepackage{color}
\usepackage{xcolor}

\theoremstyle{plain} \textheight 22.5truecm \textwidth 14.5truecm
\newcommand{\R}{\mathbb{R}}             

\def\<{\langle}
\def\>{\rangle}

\theoremstyle{plain}
\newtheorem{theorem}{Theorem}[section]
\newtheorem{lemma}[theorem]{Lemma}
\newtheorem{corollary}[theorem]{Corollary}
\newtheorem{proposition}[theorem]{Proposition}

\theoremstyle{definition}
\newtheorem*{definition*}{Definition}

\newtheorem{example}[theorem]{Example}

\theoremstyle{remark}
\newtheorem{remark}[theorem]{Remark}

\begin{document}
\title{Hypercyclicity of operators that $\lambda$-commute with the Hardy backward shift}

\author{ Mohamed Amouch, Fernando León-Saavedra,  and M.P. Romero de la Rosa }

\address{\textsc{Mohamed Amouch },
University Chouaib Doukkali.
Department of Mathematics, Faculty of science
Eljadida, Morocco}

\address{\textsc{Fernando León-Saavedra}. Universidad de Cádiz, Departamento de Matemáticas, Spain}

\address{\textsc{M.P Romero de la Rosa}.  Universidad de Cádiz, Departamento de Matemáticas, Spain}
\email{amouch.m@ucd.ac.ma}
\email{fernando.leon@uca.es}
\email{pilar.romero@uca.es}


\begin{abstract}

An operator $T$ acting on a separable complex Hilbert space $H$ is said to
be hypercyclic if there exists $f\in H$ such that the orbit $\{T^n
f:\ n\in \mathbb{N}\}$ is dense in $H$. Godefroy and Shapiro \cite{GoSha} characterized those elements in the commutant of the Hardy backward shift  which are hypercyclic.
In this paper we study some dynamics properties of operators $X$ that $\lambda$-commute with the Hardy backward shift $B$, that is, $BX=\lambda XB$. 
\end{abstract}
\maketitle

\section{Introduction}

A   bounded linear operator $T$, defined on a complex separable Banach space $X$, is said to be hypercyclic if there exists a vector $x\in X$ such that
the orbit: $\mathcal{O}(x,T):=\{T^n x:\
n\in \mathbb{N}\}$ is dense in $X$.
In this case, the vector $x$ is called a hypercyclic vector for $T$.

The main ancestors of this paper are \cite{onorbofele} and \cite{GoSha}.
In 1969 Rolewicz \cite{onorbofele} show up the first example hypercyclic operator defined on a Banach space. If we considered  $T=\lambda B$ where $B$ is the backward shift operator defined on the sequences spaces $\ell_p$, Rolewicz proved that $T=\lambda B$ is hypercyclic if and only if $|\lambda|>1$. Rolewicz’s findings were united in 1991 by G.
Godefroy and J. H. Shapiro (see \cite{GoSha}), who demonstrated that each non-scalar operator that
commutes with $B$ defined on the Hardy space $H^2(\mathbb{D})$ is hypercyclic if and only if the interior of its punctual spectrum intersects the
unit circle.  Godefroy and Shapiro's seminal work points to the idea that the hypercyclic properties of a given operators are somehow transferred to the commutator of the operator. Suprisingly, for the Bergman backward shift the commutant hypercyclicity problem is much more delicate (see \cite{bousha}).

On the other hand, we found in the literature the $\lambda$-commutant notion, that is, two operators are $\lambda$-commuting if they commute up to a complex factor $\lambda$. The term $\lambda$-commuting was introduce by Conway and Prajitura (\cite{onlamcomope}), and since this relation is not symmetric, a more precise terminology was introduced: A complex number $\lambda$ is said to be an {\it extended eigenvalue} of an operator $T$ if there exists an operator $X\neq 0$ (later called {\it extended $\lambda$-eigenoperator} of $T$) such that $TX=\lambda XT$.

Let us remark that although we have defined a hypercyclic operator on a Banach space, this property is not exclusive por Banach space theorists. In fact the first examples of hypercyclic operator were discovered a century ago defined on the  space of entire functions endowed with the compact-open topology (see \cite{lincha} for more historical notes).
Recent research is focused on to try to see how the hypercyclic properties are transferred to the {\it $\lambda$-commutant} of a given operator. For instance, some dynamics properties of the extended eigenoperators  of the differentiation operator defined on the space of entire functions was studied in \cite{BGSR,GSR,note}.  In this note, we wish to explore the dynamic properties of extended eigenoperators of the backward shift operator defined on the Hardy space $H^2(\mathbb{D})$. 
We see that in Banach spaces, the existence of the uniform norm of operators makes it difficult for hypercyclic properties to be transferred to the $\lambda$-commutant. In strong contrast, in Fréchet spaces this transfer seems to be much easier.

Let us recall that a Banach space operator $T$ is called supercyclic if there exists a vector $x\in X$ such that the set of scalar
multiples of $\mathcal{O}(x,T)$:
$$\mathbb{C}.\mathcal{O}(x,T):=\{\mu T^nx:\ \mu\in \mathbb{C},\ n\in \mathbb{N}\}$$ is dense in $X$. And let us recall that  a  linear operator $\mathcal{B}$ on a Banach space $X$ is said to be  a generalized backward shift if it satisfies the following conditions:
\begin{enumerate}
  \item The kernel of $\mathcal{B}$ is one dimensional.
\item $\bigcup\{\ker \mathcal{B}^n:\ n = 0, 1, 2,...\}$ is dense in $X$.
\end{enumerate}

We will see that  other dynamics properties, such as supercyclicity, that  is transferred very well to the  commutant, it founds some serious dificulties to be transferred to the $\lambda$-commutant. For instance,
V. M\"uller (\cite{muller}), solving a longstanding question posed by Godefroy and Shapiro in \cite{GoSha}, proved that any non-scalar operator commuting with a generalized backward shift is supercyclic. 
However this results is not longer true for the $\lambda$-commutant of the backward shift operator defined on the Hardy space $H^2(\mathbb{D})$.

The paper is structured as follows. In Section \ref{tools} we introduce the main tools that we will use throughout the paper. We will use the Hypercyclicity Criterion, that is, a sufficient condition for hypercyclicity. Next, in order to find enough intuition to address our research, we will review the result of Godefroy and Shapiro that characterizes which elements of the commutant of $B$ are hypercyclic. Specifically, 
in Godefroy and Shapiro's result we will see a dichotomy on the orbits of operators $T$ that commute with the operator $B$: the operator $T$ is either hypercyclic or the orbits of $T$ or $T^{-1}$ are bounded. In general, something similar seems to happen in the case of operators that $\lambda$-commutes with $B$.


The cornerstone in Section \ref{factorization} is a factorization of the extended $\lambda$-eigenoperators of the backward shift. This factorization complements the results obtained by S. Petrovic in \cite{Pet} and it is a main tool that will be used in the next sections.
When $|\lambda|=1$  using the ideas of \cite{LaLePeZa} we get that each extended $\lambda$-eigenoperator of the backward shift factorize as $R_\lambda\phi(B)$ when $R_\lambda f(z)=f(\lambda z)$ is the dilation operator and $\phi(B)$ is an element of the commutant of $B$, that is, the adjoint of a multiplier on $H^2(\mathbb{D})$. Surprisingly enough, when $|\lambda|<1$ the extended $\lambda$-eigenoperators have the form $R_\lambda\phi(B)$ when $\phi$ is an element of $H^2(\mathbb{D})$.

In Section \ref{hypercyclicity} we will study the Hypercyclicity of an extended $\lambda$-eigenoperator 
of the backward shift. 
Firstly, we show that if $R_\lambda \phi(B)$ is  hypercyclic then $\phi(B)$ must be hypercyclic. But the converse is not true. In strong contrast the 
with the results in \cite{BGSR}, the converse is not true,  a hypercyclic operator $\phi(B)$  in the commutant of $B$ in general it does not induce a hypercyclic extended  $\lambda$-eigenoperator. The problem depends on the instrinsic geometry of $\phi(\mathbb{D})$.
However, when $\lambda$ is a root of the unit we can get a full characterization in terms of the properties of $\phi$. 
When $\lambda$ is an irrational rotation the problem is connected to the study of the dynamical properties of a sequence of functions. 
A characterization of the hypercyclicity of $R_\lambda\phi(B)$ in terms of the geometry of $\phi(\mathbb{D})$, seems to be a challenging problem.

In Section \ref{supercyclicity} we will show that if $T=R_\lambda \phi(B)$ is an extended $\lambda$-eigenoperator of the backward shift operator with $|\lambda|<1$, then $T$ is supercyclic if and only if $\phi(0)=0$. As byproduct we obtain that M\"uller's result is not longer true for elements in the $\lambda$-commutant of the  backward shift operator. The paper is closed with a brief section with open questions and further directions.

\section{Some tools}

\label{tools}

We will use the following version of the hypercyclicity Criterion
formulated by J. Bès and A. Peris in \cite{herhypope}: 
\begin{theorem}
Let $T$ be an operator on an $F$-space $X$ satisfying the following
conditions: there exist $X_0$ and $Y_0$ dense subsets of $X$, a
sequence $(n_k)$ of non-negative integers, and (not necessarily
continuous) mappings $S_{n_k} : Y_0 \rightarrow X$ so that:
\begin{enumerate}
\item $T^{n_k} \rightarrow 0$ pointwise on $X_0$;
\item $S_{n_k} \rightarrow 0$ pointwise on $Y_0$;
\item $T^{n_k}S_{n_k} \rightarrow Id_{Y_0}$ pointwise on $Y_0$.
\end{enumerate}
Then the operator $T$ is hypercyclic.
\end{theorem}
Specifically we will use the following spectral sufficient condition discovered by Godefroy and Shapiro \cite{GoSha}.

\begin{theorem}[Godefroy-Shapiro]
\label{spectral}
Let $T$ be a bounded linear operator defined on a Banach space $X$. If  the eigenspaces $\text{ker}(T-\lambda I)$ with $|\lambda|>1$ span a dense subspace of $X$ as well as the eigenspaces $\text{ker}(T-\lambda I)$ with $|\lambda|<1$, then $T$ is hypercyclic. 
\end{theorem}

Let us denote by $k_a(z)=\frac{1}{1-\overline{a}z}$ the reproducing kernels on $H^2(\mathbb{D})$ and $M_g$ denotes the analytic Toeplitz operator with symbol $g\in H^\infty(\mathbb{D})$. 
It is well known that $M_g^\star k_a(z)=\overline{g(a)} k_{a}(z)$.
As usual, we denote by $\overline{g}(z)=\overline{g(\overline{z})}\in H^\infty(\mathbb{D})$.
By convenience, we will denote the elements of the commutant of $B$, by $\phi(B)$, with $\phi\in H^\infty(\mathbb{D})$.

The spectrum of $M_g^\star$ is usually big, in fact $M_{\overline{\phi}}^\star k_a(z)=\phi(a) k_a(z)$. Thus, using Theorem \ref{spectral}, if $\phi(\mathbb{D})$ meet the unit circle then $M_{\overline{\phi}}^\star$ is hypercyclic. 
In other cases, that is, if $\phi(\mathbb{D})\subset \mathbb{D}$ (respectively $\phi(\mathbb{D})\subset \mathbb{C}\setminus\overline{\mathbb{D}}$) then the orbit  $\|(M_\phi^\star)^nf\|\leq M$ are bounded (respectively the orbit of $f$ under  $(M_\phi^\star)^{-1}$ is bounded). This dichotomy that appears in this result will reappear in the results of this work, and it will be central in the discussion.

Let $X$ be a topological space and $(\phi_k)$ a sequence of continuous mappings on $X$. A dynamical system $(X,(\phi_k))$ is topologically transitive if for any non-empty open sets $U,V\subset X$ there exists a positive integer $n$ such that $\phi_n(U)\cap V\neq \emptyset$. 

\begin{theorem}[Birkhoff's Transitivity theorem]
\label{birkhoff}
If $X$ is a second countable, complete metric space, then topological transitivity implies that there is a dense set of points $x$ in $X$ with dense orbit $\{\phi_k(x)\}_{k\geq 1}$.
\end{theorem}

The following version of Montel's theorem  in the practice can seen as a type of Birkhoff's transitivity theorem.

\begin{theorem}[Montel's Theorem]
\label{montel}
Let us suppose that 
$\mathcal{F}$ is a family of meromorphic functions defined on an open subset $D$. If $z_0\in D$ is such that $\mathcal{F}$ is not normal at $z_0$ and $z_0\in U\subset D$, then  
$$ \bigcup_{f\in \mathcal{F} }   f(U)$$
is dense for any non-empty neighbourhood $U$ of $z_0$.
\end{theorem}

\section{Factorization of extended $\lambda$-eigenoperators of the backward shift.}
\label{factorization}

Assume that $X$ is an extended $\lambda$-eigenoperator
of the backward shift operator $B$, on $H^2(\mathbb{D})$. 
Let us discard a trivial case: $\lambda=0$.
An easy check show that if $\lambda=0$ then
$X1=c_0\neq 0$ and $Xz^n=0$ for any $n\geq 1$. That is, $X$ is a projection over the constant functions and 
therefore it is not hypercyclic.

As a consequence of a result by S. Petrovic \cite{Pet}, when $|\lambda|>1$ then there is no extended $\lambda$-eigenoperators. That is, the extended spectrum of $B$ is the closed unit disk.
Next, we state a result that factorize $X$ when it exists.

\begin{theorem}
Assume that $\lambda\in \overline{\mathbb{D}}$ and $X$ is an extended $\lambda$-eigenoperator of $B$ then:
\begin{enumerate}
    \item If $|\lambda|=1$, then $X$ is an extended $\lambda$-eigenoperator of $B$ if and only if $X=R_\lambda\phi(B)$ 
    with $\phi\in H^\infty(\mathbb{D})$.
    \item If $|\lambda|<1$ then then $X$ is an extended $\lambda$-eigenoperator of $B$ if and only if $X=R_\lambda\phi(B)$ with $\phi\in H^2(\mathbb{D})$.
\end{enumerate}
\end{theorem}
\begin{proof}
To show (1), assume that if $X=R_\lambda\phi(B)$ when $\phi\in H^\infty(\mathbb{D})$. 
Since $R_\lambda$ is an extended $\lambda$-eigenoperator of $B$,  an easy check show that $X$ is an extended $\lambda$-eigenoperator of $B$.

Conversely, let  $X$ be an extended $\lambda$-eigenoperator of $B$.
Since $R_{\lambda^{-1}}$ is an extended $1/\lambda$-eigenoperator of $B$, then  $BR_{1\over \lambda}= {1\over \lambda} R_{1\over \lambda}B,$
hence
$$(R_{1\over \lambda}X)B= {1 \over \lambda} R_{1 \over \lambda}BX
={1 \over \lambda} {\lambda}
BR_{1 \over \lambda}X= B(R_{1 \over \lambda}X),$$
that is $B(R_{1\over \lambda}X)$ commutes with $B$.
So, there exists 
$\phi\in H^\infty(B)$  such that  $R_{1 \over \lambda}X
= \phi(B)$ therefore 
$ X= R_{ \lambda} \phi(B)$ as we desired.

To show (2), we will use a result by S. Petrovic (see \cite{Pet}). Specifically, Petrovic proved that if $X$ is an extended $\lambda$-eigenoperator of $B$ ($|\lambda|<1$) then
$$
X=R_\lambda\phi(B)
$$
where $\phi(z)=\sum_{k=0}^\infty c_n z^n$ is a non-zero formal power series.
Taking adjoints, we get that $X^\star=M_{\overline{\phi}} R_{\overline{\lambda}}$ where $M_{\overline{\phi}}$ could be an unbounded multiplication operator. Let us observe that
$X^\star 1=\overline{\phi}\in H^2(\mathbb{D})$ which proves the first implication.

Conversely, assume that $\phi\in H^2(\mathbb{D)}$, and $|\lambda|<1$ we have to show that 
$M_{\overline{\phi}} R_{\overline{\lambda}}$ is bounded. Indeed, for any $f\in H^2(\mathbb{D})$. Since $|\lambda|<1$  the map
$R_{\overline{\lambda}}f$ send $H^2(\mathbb{D})$ on $H^\infty(\mathbb{D})$. Thus, since $R_{\overline{\lambda}}f\in H^\infty(\mathbb{D})$ we get:
\begin{equation}
\label{uno}
\|M_{\overline{\phi}} R_{\overline{\lambda}} f\|_2^2\leq \|\phi\|_2^2 \max_{|w|=|\lambda|}|f(w)|^2
\end{equation}
By using the maximum modulus principle, we get $\max_{|w|=|\lambda|}|f(w)|^2=|f(w_0)|^2$ for some $|w_0|=|\lambda|$. Using the subharmonicity of $|f|^2$ , we get that there exists a constant $C>0$ such that for any $r>0$ such that $|\lambda|<r<1$ we get that
\begin{equation}
    \label{dos}
|f(w_0)|^2\leq C  \frac{1}{2\pi}\int_{0}^{2\pi} |f(re^{i\theta}|^2\,d\theta.
\end{equation}

Thus, using inequality (\ref{dos}) into inequality (\ref{uno}) we get
$$
\|M_{\overline{\phi}} R_{\overline{\lambda}} f\|_2^2\leq C \|\phi\|_2^2 \|f\|_2^2,
$$
which gives the desired result.

\end{proof}

\section{Hypercyclicity of extended $\lambda$-eigenoperators of the backward shift.}

\label{hypercyclicity}

In this section we study when a $\lambda$-extended eigenoperator of $B$, is hypercyclic. 

First of all,  we will see some basics consequences from the theory of hypercyclic operators. Next, we will study some sufficient conditions for hypercyclicity.

\subsection{Hypercyclicity: basic consequences.}

It is known that for Banach space operators, extended $\lambda$-eigenoperators of a given operator $A$ are never hypercyclic provided $|\lambda|<1$ (see \cite[Proposition 3.1]{BGSR}). Thus we get:

\begin{corollary}
Assume that $T$ is an extended $\lambda$-eigenoperator of $B$ with $|\lambda|<1$ then $T$ is not hypercyclic.
\end{corollary}

The following example provides some intuition  about the last statement in our special case.
\begin{example}
If we consider Rolewicz operator $2B$, the point spectrum of $2B$ is big. If we consider $T=R_\lambda (2B)$ with $|\lambda|<1$, we get that $R_\lambda$ is compact, therefore $T$ is compact. Hence $T$ cannot be hypercyclic.
\end{example}

We stress here that $T=R_\lambda\phi(B)$ with $|\lambda|<1$ is not hypercyclic even if $\phi(B)$ is unbounded. However $T$ could be supercyclic. In fact, in Section \ref{supercyclicity} we characterize when $R_\lambda\phi(B)$ is supercyclic in the case $|\lambda|<1$.


Let us assume that $|\lambda|=1$. In such a case, our extended $\lambda$-eigenoperators factorize as $T=R_\lambda\phi(B)$, with $\phi\in H^\infty(\mathbb{D})$. Then, if $T$ is hypercyclic then $\phi(B)$ is hypercyclic too.

\begin{proposition}
Assume that $T=R_\lambda\phi(B)$ with $|\lambda|=1$ and $\phi\in H^\infty(\mathbb{D})$. If $\phi(\mathbb{D}) \cap \partial \mathbb{D}=\emptyset$ then $T$ is not hypercyclic.
\end{proposition}
\begin{proof}
   Indeed, if $\phi(\mathbb{D})\subset \mathbb{D}$ then since $\|T\|\leq \|R_\lambda\|\|\phi\|_{\infty}$ we get that $T$ is a contraction, therefore $T$ is not hypercyclic.

  On the other hand, if $\phi(\mathbb{D})\subset \overline{\mathbb{D}}^c$, then we see that $T$ is invertible and $\|T^{-1}\|\leq 1$, therefore  $T^{-1}$ is not hypercyclic. Hence $T$ is not hypercyclic.\end{proof}

Assume that $T=R_\lambda\phi(B)$, with $|\lambda|=1$ and $\phi(B)$ hypercyclic. In strong contrast with the results obtained in \cite{BGSR}, if $\phi(B)$ is hypercyclicity then we cannot guarantee that $T$ is hypercyclic too.  The next example provides a hypercyclic operator $\phi(B)$, with $\phi\in H^\infty(\mathbb{D})$, such that  for some $\lambda$ with $|\lambda|=1$, the extended $\lambda$-eigenoperator $R_\lambda \phi(B)$  is not hypercyclic.

\begin{example}
Let us consider $\lambda=i$ the imaginary unit and the maps
$$
\phi_0(z)=\frac{1}{2}z+1-\frac{1}{10}.
$$
and
$$
\phi_1(z)=\frac{1}{2}z+1-\frac{1}{100}.
$$
Clearly $\phi_0(\mathbb{D})$ is the disk centered at $1-\frac{1}{10}$ with radius $1/2$ and
$\phi_1(\mathbb{D})$ is the disk centered at $1-\frac{1}{100}$ with radius $1/2$ (see Figure \ref{fig1}).

\begin{figure}[h]
\centering
\includegraphics[scale=0.5]{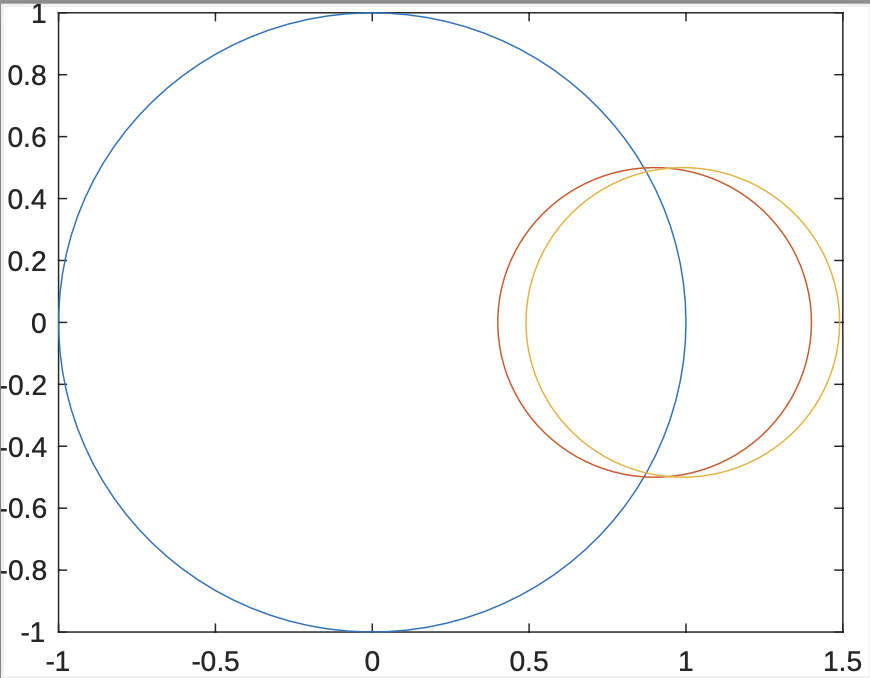}
\caption{The image $\phi_0(\mathbb{D})$ and $\phi_1(\mathbb{D})$.}
\label{fig1}
\end{figure}
Thus $\phi_0(\mathbb{D})\cap \partial \mathbb{D}\neq \emptyset$ and
$\phi_1(\mathbb{D})\cap \partial \mathbb{D}\neq \emptyset$,
 therefore $\phi_0(B)$ and $\phi_1(B)$ are hypercyclic.
 
Let us see consider the operators $T_0=R_{i}\phi_0(B)$ and $T_1=R_{i}\phi_1(B)$,
and the powers $T^4_0$ and $T_1^4$. Since $i^4=1$ we obtain that $T_0^4$ and $T_1^4$ commute with $B$. Specifically, $T_0^4=\phi_0(B)\phi_0(iB)\phi_0(-iB)\phi_0(-B)$ and
 $T_1^4=\phi_1(B)\phi_1(iB)\phi_1(-iB)\phi_1(-B)$
Let us denote 
$\Phi_0(z)=\phi_0(z)\phi_0(iz)\phi_0(-iz)\phi_0(-z)$ and $\Phi_1(z)=\phi_1(z)\phi_1(iz)\phi_1(-iz)\phi_1(-z)$. 

Now, using Matlab computations, we see that $\Phi_0(\mathbb{D})\subset \mathbb{D}$
and $\Phi_1(\mathbb{D}) \cap \partial \mathbb{D}\neq \emptyset$ (see Figure \ref{fig2}).

\begin{figure}[h]
\centering
\includegraphics[scale=0.8]{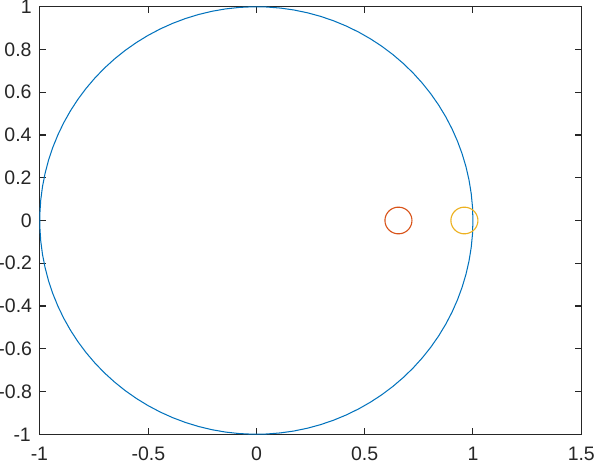}
\caption{The image $\Phi_0(\mathbb{D})$ and $\Phi_1(\mathbb{D})$  .}
\label{fig2}
\end{figure} 
Therefore $T_0^4$ is not hypercyclic and $T_1^4$ is hypercyclic. Hence, $T_0$ is not hypercyclic, but $T_1$ is hypercyclic by Ansari's result (\cite{ansari}).

\end{example}
From the above considerations we see that the hypercyclicity of $R_\lambda\phi(B)$   does not depend generally on intersection of the point spectrum of $\phi(B)$ at $\partial \mathbb{D}$, but it clearly depends of  the shape of the point spectrum of $\phi(B)$. This shape is intimatelly connected to  the limits of the sequence of sets $\Phi_{n}(\mathbb{D})$ when $\Phi_{n}(z)=\phi(z)\phi(\lambda z)\cdots\phi(\lambda^{n-1}z)$. Specifically, if the sequence $\Phi_{n}(\mathbb{D})$ is separated from $\infty$ or from $0$ then $T$ is not hypercyclic.
\begin{proposition}
\label{acotacion}
Assume that $\lambda\in\partial \mathbb{D}$ and $T=R_{\lambda}\phi(B)$ is an extended $\lambda$-eigenoperator of $B$.
\begin{enumerate}
\item If $\|\Phi_{n}\|_\infty<M$ for all $n$ then $T$ is not hypercyclic.
\item If there exists $c>0$ such that $c\leq Inf_{z\in \mathbb{D}}|\Phi_n(z)|$ for all $n\geq n_0$ for some $n_0$, then $T=R_{\lambda}\phi(B)$ is not hypercyclic.
\end{enumerate}
\end{proposition}
\begin{proof}
To show (1), let us see that 
\begin{eqnarray*}
\|(R_{\lambda}\phi(B))^n\|&=&|1\cdot\lambda \cdots\lambda^{n-1}| \|R_\lambda^n\|\|\Phi_n(B)\| \\
&\leq & \|\Phi_n\|_{\infty}\leq M,
\end{eqnarray*}
that is $T=R_{\lambda}\phi(B)$ is power bounded, and therefore $T$ is not hypercyclic.

For (2), let us observe that $T^n$ is invertible for $n\geq n_0$. Since $T$ is hypercyclic if and only if $T^n$ is hypercyclic, we get that  $T$ is hypercyclic if and only if $T^{-n}$ is hypercyclic. However, and easy check show that
$
\|T^{-n}\|\leq \frac{1}{c}
$
for all $n\geq n_0$. Therefore $T^{-n}$ with $n\geq n_0$, is not hypercyclic. Hence $T$ is not hypercyclic.
\end{proof}

However, for values $\lambda\in\partial\mathbb{D}$ which are roots of the unity, a characterization can be obtained using Ansari and Godefroy-Shapiro results.

\begin{proposition}
 Assume that $T=R_\lambda\phi(B)$ is an extended $\lambda$-eigenoperator of $B$ and $\lambda^n=1$ for some positive integer $n$. If we denote $\Phi(z)=\phi(z)\phi(\lambda z)\cdots \phi(\lambda^{n-1}z)$, then $T$ is hypercyclic if and only if $\Phi(\mathbb{D})\cap\partial\mathbb{D}\neq \emptyset.$
\end{proposition}
 \begin{proof}
 Let us see that
 $T^n=R_\lambda \phi(B)\cdots R_{\lambda}\phi(B)$, and since $R_\lambda$ is an extended $\lambda$-eigenoperator of $B$ we have
 $ \phi(B)R_\lambda=\phi(\lambda B)R_\lambda $. Thus
 $T^n=(R_\lambda)^n \phi(\lambda^{n-1}B)\phi(\lambda^{n-2}B)\cdots \phi(B)$. Now, since $\lambda^n=1$ we get $(R_\lambda)^n=R_{\lambda^n}=I$. Therefore, since $T^n=\Phi(B)$ is an element of the commutant of $B$, using Godefroy-Shapiro characterization,
 $T^n$ is hypercyclic if and only if $\Phi(\mathbb{D})\cap \partial \mathbb{D}\neq \emptyset$. Hence, using Ansari's result we get that $T$ is hypercyclic if and only if $\Phi(\mathbb{D})\cap \partial \mathbb{D}\neq \emptyset$ as we desired.

 \end{proof}

To simplify the notation, set $\alpha=\overline{\lambda}$ and  $\omega=1/\overline{\lambda}$.
Thus, the existence of hypercyclic extended $\lambda$-eigenoperators will depends on the dynamics of the sequence of analytic functions
$$\Psi_n(z)=\overline{\phi}(z) \cdot \overline{\phi}(\alpha z)\cdots \overline{\phi}(\alpha^{n-1}z)$$
and
$$
\Omega_n(z)=\overline{\phi}(\omega z)\cdots \overline{\phi}(\omega^{n}z).
$$
In fact, the dynamics of both sequences are related by the formula $\Omega_n(z)=\Psi_n(\omega^nz)$. The reader keep in mind that $\Psi_n(z)\to \infty$ (resp. $0$) if and only if $\overline{\Psi_n(z})\to \infty$ (resp. $0$), and the same property is satisfied for the sequence of mappings $\Omega_n(z)$.
The next result sheds light on this direction.

\begin{proposition}
Assume that $\lambda\in \partial \mathbb{D}$ is an irrational rotation. If there exists a sequence $(n_k)\subset \mathbb{N}$ such that
$$
C_0=\{z\in \mathbb{D}\,:\,\Psi_{n_k}(z)\to 0\}
$$
and
$$
C_1=\{z\in \mathbb{D}\,:\, \Omega_{n_k}(z)\to \infty\}
$$
have both an accumulation point on $\mathbb{D}$, then $T=R_\lambda\phi(B)$ is hypercyclic on $H^2(\mathbb{D})$.
\end{proposition}
\begin{proof}
Indeed, we can see that the Hypercyclicy Criterion is satisfied by selecting the dense subsets
$$
X_0=\textrm{linear span}\{k_{z_0}(z)\,:\,z_0\in C_0\}
$$
and
$$
Y_0=\textrm{linear span}\{k_{z_0}(z)\,:\,z_0\in C_1\}.
$$
Both subsets $X_0, Y_0$ are dense and
$$
T^{n_k}k_{z_0}(z)=\overline{\Psi_{n_k}(z_0)}k_{\alpha^{n_k}z_0}\to 0
$$  
pointwise on $X_0$. On the other hand,  if $a\in C_1$ then  $\overline{\phi}(\omega^n a)\neq 0$ we can define
$$
S k_a(z)=\frac{1}{\overline{\overline{\phi}(\omega a)}} k_{\omega a}(z)
$$
then $S$ is well defined on $Y_0$ and
$S^{n_k}$ converges pointwise to $0$ on $Y_0$, then $T$ is hypercyclic.
\end{proof}

The above result is a natural way to apply the Hypercyclicity Criterion. However, the hyphotesis are very restrictive and next we will get some weakening.

\begin{proposition}
\label{unique}
Assume that $\lambda\in \partial \mathbb{D}$ is an irrational rotation and $\phi\in H^\infty(\mathbb{D})$. 
\begin{enumerate}
\item If there exist $z_0\in \mathbb{D}$ and a subsequence $(n_k)\subset \mathbb{N}$ such that
$\Psi_{n_k}(z_0)\to 0$, then $T^{n_k}$ converge pointwise to zero on a dense subset.

\item If there exists $z_1\in \mathbb{D}$ such that $\Omega_{n_k}(z_1)\to \infty$ and there exists $m>0$
such that $|\phi(\omega^{n}z_0)|>M>0$ for all $n\geq 1$, then the right inverse $Sk_a=\frac{1}{\overline{\overline{\phi}(\omega a)}}k_{\omega a}(z)$ is well defined on a dense subset $Y_0$ and $S^{n_k}$ converges pointwise to zero on $Y_0$.
\end{enumerate}
\end{proposition}
\begin{proof}
Indeed, if $\overline{\phi}(\alpha^{n_0}z_0)=0$ for some $n_0\in \mathbb{N}$ then let us take
$$
X_{0}=\textrm{linear span} \{k_{\omega^nz_0}(z)\,\,:\,\,n\in\mathbb{N}\}
$$
Since $\omega$ is an irrational rotation, the subset $\{\omega^nz_0\}$ has an accumulation point in $\mathbb{D}$, which implies that $X_0$ is dense. Let us fix $k_0\in \mathbb{N}$, if $n>n_0+k_0$ then
$$
T^n k_{\omega^{k_0}z_0}(z)=0.
$$
That is, $T^{n}$ converges pointwise to zero on $X_0$ for the full sequence of natural numbers. In particular $T^{n_k}$
converges pointwise to zero on $X_0$.

If $\overline{\phi}(\alpha^{n}z_0)\neq 0$ for all $n\in\mathbb{N}$ then we consider the subset
$$
X_{0}=\textrm{linear span}\{k_{\alpha^n z_0}(z)\,:\,n\geq 1\}.
$$
Again, since $\alpha$ is an irrational rotation, the subset $\{\alpha^n z_0\}$ has an accumulation point in $\mathbb{D}$, therefore the subset $X_0$ is  dense. On the other hand
\begin{eqnarray*}
T^{n_k}k_{\alpha^{n_0}z_0} (z)&=& \overline{\overline{\phi}(\alpha^{n_0}z_0)\cdots \overline{\phi}(
\alpha^{n_0+n_{k}-1}z_0)} k_{\alpha^{n_0+n_k}z_0}(z)\\
&=& \frac{\overline{\Phi_{n_k}(z_0)}\overline{\overline{\phi}(\alpha^{n_k}z_0)\cdots\overline{\phi}(\alpha^{n_0+n_k-1}z_0)} }{\overline{\overline{\phi}(z_0)\cdots \overline{\phi}(\alpha^{n_0-1}z_0)}} k_{\alpha^{n_0+n_k}z_0}(z),
\end{eqnarray*}

Thus,
$$
\|T^{n_k}k_{\alpha^{n_0}z_0} (z)\|\leq C |\Phi_{n_k}(z_0)|\|\phi\|_{\infty}^{n_0}.
$$
We obtain again that $T^{n_k}$ converges pointwise to zero on $X_0$ as $k\to\infty$. This fact proves (1).

To show (2), we consider  the following subset
$$
Y_{0}=\textrm{linear span}\{k_{\omega^nz_1}(z)\,:\,n\geq 1\}
$$
which is clearly dense because the sequence $\omega^nz_1$ has an accumulation point in $\mathbb{D}$.
Next we will see that on $Y_0$ the following map 
$$
Sk_a(z)=\frac{1}{\overline{\overline{\phi}(\omega a)}} k_{\omega a}(z)
$$
is well defined, because $\Omega_{n_k}(z_1)\to \infty$ therefore we get that
$\overline{\phi}(\omega^nz_1)\neq 0$ for all $n$.

On the other hand
\begin{eqnarray*}
S^{n_k}k_{\omega^{m_0}z_1}&=&\frac{1}{\overline{\overline{\phi}(\omega^{m_0+1}z_1)}\cdots \overline{\overline{\phi}(\omega^{m_k+m_0}z_1)}}k_{\omega^{m_k+m_0}z_1}(z)\\ 
&=& \frac{\overline{\overline{\phi}(\omega z_1)}\cdots\overline{\overline{\phi}(\omega^{m_0}z_1)}}{\overline{\Omega_{n_k}(z_1)} \overline{\overline{\phi}(\omega^{m_k+1}z_1)}\cdots \overline{\overline{\phi}(\omega^{m_k+m_0}z_1)}},
\end{eqnarray*}
Therefore
$$
\|S^{n_k}k_{\omega^{m_0}z_1}\|\leq \frac{\|\varphi\|_\infty^{m_0}}{M^{m_0}|\Omega_{n_k}(z_1)|}\to 0 
$$
as $k\to \infty$, as we wanted.
\end{proof}

From the above results, now we wish to obtain some geometrical sufficient conditions to guarantee hypercyclicity. Since we will apply the Hypercyclicity Criterion, we will look sufficient conditions to obtain separately the conditions of the Hypercyclicity Criterion. We will say that $T$ satisfies the condition $X_0$ (respectively $Y_0$) if $T$ satisfies the condition (1) (respectively (2)) of the Hypercyclicity Criterion. 

\begin{proposition}
\label{sufficient}
Assume that $T=R_\lambda \phi(B)$ is an extended $\lambda$-eingenoperator of $B$. Then
\begin{enumerate}
\item If $\phi(z_0)=0$ then $T$ satisfies the condition $X_0$ in the Hypercyclicity Criterion for the full sequence of natural numbers.
\item If $|\phi(0)|<1$ then  $T$ satisfies the condition $X_0$ is satisfied for the full sequence of natural numbers.
\item If $|\phi(0)|>1$ then $T$ satisfies the condition $Y_0$ for the full sequence of natural numbers.
\item If $|\phi(0)|=1$ and $\lambda$ is a root of the unity then $T$ satisfies conditions $X_0$ and $Y_0$ for the full sequence $\mathbb{N}$.
\end{enumerate}
\end{proposition}
\begin{proof}
Condition (1) appears in the proof of Proposition \ref{unique}. If $z_0=0$ and $\phi(0)=0$, then it is easy to verify that $T$ has dense generalized kernel. If $z_0\neq 0$, since $0=\phi(z_0)=\overline{\phi}(\overline{z_0})$ considering the following dense subset:
$$
X_0=\textrm{linear span}\{ k_{\omega^n\overline{z_0}}(z)\,:\,n\geq 1\},
$$
an easy check proves that $T$ converges pointwise to $0$ on $X_0$.

To show (2), by continuity there exists $\delta>0$ such that   $|\phi(z_0)|<1$ for every $z_0\in D(0,\delta)$. The result follows by considering the following dense subset:
$$
X_0=\textrm{linear span} \{k_{\overline{z_0}}(z)\,\,:\,\,z_0\in D(0,\delta)\}.
$$
In a similar way there exists $\delta>0$ such that for each $z_0\in D(0,\delta)$, $|\phi(z_0)|>1+\varepsilon$ for some $\varepsilon>0$. Now condition $(3)$ follows by considering:
$$
Y_0=\textrm{linear span}\{k_{\overline{z_0}}(z)\,;\,z_0\in D(0,\delta)\}.
$$
To show (4) we use the open mapping theorem. We know that if we consider $T^n$ then $T^n=\Phi(B)$ and $\Phi(z)=\phi(z)\cdot \phi(\lambda z)\cdots\phi(\lambda^{n-1}z)$ is analytic on $\mathbb{D}$. Since $|\Phi(0)|=1$ and $\Phi(0)$ is an interior point in the image $\Phi(\mathbb{D})$ we get that $\Phi(\mathbb{D})\cap \partial \mathbb{D}\neq \emptyset$, therefore by Godefroy and Shapiro's result we get that $T$ satisfies the Hypercyclicity Criterion for the full sequence of natural numbers. 
\end{proof}

\subsection{Some sufficient conditions of hypercyclicity of the elements  of the $\lambda$-commutant of the Hardy backward shift.}

In this subsection we will analyze the case $\lambda\in \partial \mathbb{D}$ and $\lambda$ an irrational rotation.

Now, let us select some of the ideas used in  Proposition 6.1 in \cite{BGSR}.

Clearly, can select  $r_n<1$ such that $r_n\to 1$ and $\varepsilon_n>0$ such that for 
 all $z$ satisfying $r_n-\varepsilon_n<|z|<r+\varepsilon_n$  we can get  $\phi(z)\neq 0$. Let us denote by $C_n=\{z\,:\,r_n-\varepsilon_n<|z|<r_n+\varepsilon_n\}$.
\begin{lemma}
\label{lemma}
Let us fix an increasing subsequence $\{n_k\}\subset \mathbb{N}$, and let us assume that for any $z_0$, satisfying $|z_0|=r_n$, the family $\mathcal{G}=\{\Omega_{n_k}\}$ is normal at $z_0$ and pointwise bounded on $z_0$. Then $\mathcal{G}$ is uniformly bounded on compact subsets of $\mathbb{D}$.
\end{lemma}
\begin{proof}
Indeed, let us fix $z_0\in C_n$ with $|z_0|=r_n$. By assumption there exists $\varepsilon>0$ such that $B_0=D(z_0,\varepsilon)\subset C_n$ and  $\mathcal{G}$ is uniformly bounded on $B_0$.  Let us shown that $\mathcal{G}$ is uniformly bounded on $|z|\leq r_n$. Indeed, since $\omega$ is an irrational rotation, by  compactness  there is an integer $n_0$  such that
$$\partial D(0,|z_0|)\subset B_0\cup \omega B_0\cup\cdots \cup\omega^n_0B_0.$$

Since $A=B_0\cup\cdots \cup \omega^{n_0}B_0$ is strictly contained in $\{z \,:\, r_0 <|z|<r_1 \}$ and $\phi$ does not vanish on $A$, let $C=\frac{\max\{|\phi(z)| : |z|=r_n\}}{\min\{|\phi(z)| : |z|=r_n\}}>0$.

 By using the modulus maximum principle and by showing that $\mathcal{G}$ is uniformly bounded on $A=B_0\cup \cdots \cup \omega^{n_0}B_0$, we get that  $\mathcal{G}$ is uniformly bounded on $D(0,|z_0|)$.

Indeed, since $\mathcal{G}$ is uniformly bounded on $B_0$, let $M$ such that $|\Omega_{n_k}(z)|<M$ for all $z\in B_0$ and $\Omega_{n_k}\in \mathcal{G}$. Let us show that $\mathcal{G}$ is uniformly bounded on each $\omega^lB_0$, $1\leq l\leq n_0$. Indeed each element in $\omega^lB_0$ have the form $\omega^lz$
with $z\in B_0$. Thus,
\begin{eqnarray*}
|\Omega_{n_k}(\omega^l z)|&= &|\overline{\phi}(\omega^{l+1}z)\cdots \overline{\phi}(\omega^{l+n_k}z)|\\
&=& \frac{|\Omega_{n_k}(z)| |\overline{\phi}(\omega^{n_k+1}z)\cdots \overline{\phi}(\omega^{n_k+l}z)|}{|\overline{\phi}(\omega z)\cdots \overline{\phi}(\omega^l z)|} \\
&\leq & M C^l,
\end{eqnarray*}
for all $\omega^l z\in \omega^l B_0$ and $\Omega_{n_k}\in\mathcal{G}$. Therefore $\mathcal{G}$ is uniformly bounded on $D(0,|z_0|)$. Since this argument is for any $r_n$ converging to $1$, we get that 
$\mathcal{G}$ is uniformly bounded on compact subsets of $\mathbb{D}$
as we wanted to show. 
\end{proof}

\begin{theorem}
\label{main}
Set $\lambda\in \partial \mathbb{D}$ an irrational rotation. If there exist a sequence $(n_k)$, a point $z_0\in \mathbb{D}$ such that $|\Omega_{n_k}(z_0)|>c>0$ and a point $z_1\in \mathbb{D}$ such that $\Psi_{n_k}(z_1)\to 0$, then $T=R_\lambda\phi(B)$ is hypercyclic.
\end{theorem}
\begin{proof}

First of all,  if $\phi$ vanishes on $\mathbb{D}$, the by applying we get that $T$ satisfies condition $X_0$ for the full sequence of natural numbers.

Let us assume that $\phi$ does not vanishes on $\mathbb{D}$. By hyphotesis
Using the same trick as in Lemma \ref{lemma}, we can get that $\Psi_{n_k}(\alpha^l z_1)\to 0$. If $|z_1|=r$ then we denote by $C_r$ an annulus like in Lemma \ref{lemma}. And let us denote $C=\frac{\max\{|\phi(z)| : z\in C_r\}}{\min\{|\phi(z)| : z\in C_r|\}}>0$ then:
\begin{eqnarray*}
|\Psi_{n_k}(\alpha^l z_1)|&= &|\overline{\phi}(\alpha^{l}z_1)\cdots \overline{\phi}(\alpha^{l+n_k-1}z_1)|\\
&=& \frac{|\Psi_{n_k}(z_1)| |\overline{\phi}(\alpha^{n_k}z_1)\cdots \overline{\phi}(\alpha^{n_k+l-1}z_1)|}{|\overline{\phi}(z_1)\cdots \overline{\phi}(\alpha^{l-1} z_1)|} \\
&\leq &|\Psi_{n_k}(z_1)| C^l\to 0.
\end{eqnarray*}
Then by applying Proposition \ref{unique} (1), the operator $T=R_\lambda\phi(B)$ satisfies condition $X_0$ of the Hypercyclicity Criterion for the sequence of natural numbers $(n_k)$.
Therefore, in both cases we obtain that $T$ satisfies condition $X_0$ for the sequence $(n_k)$.

We set  $\mathcal{G}=\{\Omega_{n_k}\}$.
Let us consider the annulus $C_n$ of the previous lemma. Then we have two possibilities:

{\it Case 1. } The family $\mathcal{G}$ is normal at no point $z\in C_{n_0}$ for some $n_0$. Then  by using Montel's Theorem (Theorem \ref{montel}) we have that $\bigcup_{l}\Omega_{n_l} (C_{n_0})$ is dense in $\mathbb{C}$. Since $C_{n_0}$ is homemorphic to a complete metric space, using Birkhoff's transitivity Theorem (Theorem \ref{birkhoff}) there exists  $z_2\in C_{n_0}$ such that
$\{\Omega_{n_l}(z_2)\}_{l\geq 1}$ is dense in $\mathbb{C}$. In particular, there exists a subsequence $\{r_k\}\subset \{n_k\}$ such that $\Omega_{r_k}(z_2)\to \infty$. On the other hand, since $z_2\in C_{n_0}$, we obtain that 
$|\phi(\omega^nz_2)|\geq \min_{C_{n_0}}|\phi(z)|>0$. Thus, the conditions of Proposition \ref{unique} (2) are satisfied. Hence, $T$ satisfies the condition $Y_0$ of the Hypercyclicity Criterion  for the subsequence $(r_k)$.

As a consequence, the conditions $X_0$ and $Y_0$ or the Hypercyclicity Criterion are satisfy for the sequence $(r_k)$, therefore $T$ is hypercyclic.

{\it Case 2. } For each $n_0$ there exists a point $z_2$ such that $\mathcal{G}$ is normal at $z_2$. We have two possibilities again. If the orbit $\{\Omega_{n_l}(z_2)\}_{l\geq 1}$ is unbounded, then there exist a subsequence $(r_k)\subset \{n_k\}$ such that $\Omega_{r_k}(z_2)\to\infty$.
In a similar way, using Proposition \ref{unique} (2) we obtain that $T$ satisfies the condition $Y_0$ of the Hypercyclicity Criterion for the sequence $(r_k)$.  Thus, $T$ satisfies the Hypercyclicity Criterion for the sequence $(r_k)$, which we wanted to prove. 

Now we suppose that for any subsequence $(n_k)$ the family $\mathcal{G}=\{\Omega_{n_k}\}$ is normal at $z_0$ and $\{\Omega_{n_k}(z_0)\}$ is bounded, then by applying Lemma \ref{lemma} we get that $\Omega_{n_k}(z)$ is uniformly bounded on compact subsets of $\mathbb{D}$. Therefore, by Montel's theorem, there exist a subsequence $(r_k)$ and an analytic function on  $\mathbb{D}$,  $g$, such that $\Omega_{r_k}\to g$ uniformly on compact subsets of $\mathbb{D}$.

Since $\Omega_n(z)=\Psi_n(\omega^nz)$, we obtain also that $\Psi_{n_l}(z)$ is uniformly bounded on compact subsets of $\mathbb{D}$. Moreover, extracting a subsequence if it is necessary, there exist a subsequence $\{r_k\}\subset\{n_k\}$, an analytic function $g\in H(\mathbb{D})$ and $\mu\in \partial\mathbb{D}$ such that  
$\Omega_{r_k}(z)\to g(z)$ and $\Psi_{r_k}(z)\to g(\mu z)$ uniformly on compact subsets of $\mathbb{D}$.

Since $\Psi_{n_k}(\alpha^l z_1)\to 0$ for all $l\geq 1$, we get that $g=0$.
On the other hand, since $|\Omega_{n_k}(z_0)|>c$, we get that $g(z_0)\neq 0$. A contradiction. Therefore, we don't have the situation that $\Omega_{n_k}$ is uniformly bounded on compact subsets of $\mathbb{D}$, and therefore $T$ is hypercyclic.
\end{proof}

\begin{corollary}
\label{corolario1}
Assume that $\lambda\in\partial \mathbb{D}$ is an irrational rotation. If $|\phi(0)|\geq 1$ and $\phi$ has a zero on $\mathbb{D}$ then $R_\lambda\phi(B)$ is hypercyclic.
\end{corollary}
\begin{proof}
Indeed, if $\varphi(a)=0$ for some $a\in \mathbb{D}$ then by applying Proposition \ref{sufficient} (1) we get that $T=R_\lambda \phi(B)$ satisfies condition $X_0$ in the Hypercyclicity Criterion for the full sequence of natural numbers.
If $|\phi(0)|\geq 1$ then $|\Omega_n(0)|=|\phi(0)|^n\geq 1$, therefore by appying Theorem \ref{main}, we get that $T$ is hypercyclic.
\end{proof}
Moreover we can obtain the same result by replacing  the existence of a zero of $\phi$ by pointwise convergence to zero for some subsequence.

\begin{corollary}
\label{corolariocaso}
Assume that $\lambda\in\partial \mathbb{D}$ is an irrational rotation. If $|\phi(0)|\geq 1$ and there exists a subsequence $(n_k)$ such that $\Psi_{n_k}(z_0)\to 0$ for some $z_0\in \mathbb{D}$ then $R_\lambda\phi(B)$ is hypercyclic.
\end{corollary}

\begin{example}
For $0< p<1$ let us consider the family of automorphisms $\varphi_p(z)$ defined by 
$$
\varphi_p(z)=\frac{p-z}{1-pz}.
$$
And we consider the family of bounded analytic functions $\psi_p(z)=\varphi_p(z)+1-p$. We see that $\psi_p(0)=1$ and $\phi_p(z)$ vanished on $\frac{2p-1}{p^2-p+1}\in \mathbb{D}$. Thus, according to Corollary \ref{corolariocaso} we get that $R_\lambda\psi_{p}(B)$ is hypercyclic for all $0< p<1$ and for all $\lambda\in\partial \mathbb{D}$.
 \end{example}

If $\lambda$ is an irrational rotation and   $|\phi(0)|<1$ in many cases, we can obtain supercyclicity.

\begin{corollary}
\label{case2}
Set $\lambda\in \partial \mathbb{D}$ an irrational rotation. If  $|\phi(0)|<1$ and $\phi$ has a zero on $\mathbb{D}$. If $T$ is not hypercyclic, then $T$ is supercyclic.
\end{corollary}
\begin{proof}
Indeed, it is sufficient to choose a constant $c>0
$ such that $|c\phi(0)|\geq 1$. Then $T=R_\lambda c\phi(B)$, is an extended $\lambda$ eigenoperator of $B$ which satisfies the hypothesis of Theorem \ref{main}. Therefore $cT$ is hypercyclic, which implies that $T$ is supercyclic.
\end{proof}

We will separate the case $T=R_\lambda \phi(B)$ invertible and will now proceed to analyze the question according to values of $\phi$ at the origin.
In the next result, we will remove from the discussion the case in which the sequences $\{\Omega_n(z)\}$ or  $\{\Psi_n(z)\}$ are not uniformly bounded on compact subsets of  $\mathbb{D}$.  

\begin{theorem}
\label{main2}
Assume that $\lambda\in \partial \mathbb{D}$ an irrational rotation:
\begin{enumerate}
\item If $\{\Omega_n\}$ is not uniformly bounded on compact subsets and $|\phi(0)|<1$ for some $z_1\in \mathbb{D}$, then $T$ is hypercyclic
\item If $\{1/\Phi_n\}$ is not uniformly bounded on compact subsets and  $|\phi(0)|>1$, then $T$ is hypercyclic.
\end{enumerate}
\end{theorem}
\begin{proof}
 Let us to show 1).  Firstly, if $|\phi(0)|<1$ by applying Proposition \ref{sufficient} (2)  we get that $T$
 satisfies the condition $X_0$ of the Hypercyclicity Criterion for the full sequence of natural numbers.

Since $\{\Omega_n\}$ is not uniformly bounded on compact subsets of $\mathbb{D}$, then we have two possibilities:

a) The family $\mathcal{G}=\{\Omega_n\}$ is normal at no  point of $\mathbb{D}$. In this case by applying Theorems \ref{birkhoff} and \ref{montel}, there is a complex number $z_1\in \mathbb{D}$ with $\mathcal{G}$-orbit dense. In particular, we can select a subsequence $\{n_k\}$ such that $\Omega_{n_k}(z_1)\to \infty$. 
Moreover, using the same trick as in Lemma \ref{lemma}, we can get that $\Omega_{n_k}(\alpha^l z_1)\to \infty$. Therefore, the condition $Y_0$ of the hypercyclicity criterion is satisfies for the subsequence $(n_k)$. As a consequence $T$ satisfies the Hypercyclicity Criterion for the sequence $(n_k)$.

b) There exists a point $z_1\in \mathbb{D}$ such that
$\Omega_{n_k}(z_1)\to \infty$. In such a case, we get again that $T$ satisfies the Hypercyclicity Criterion for the sequence $(n_k)$, and therefore $T$
is hypercyclic as we wanted to show.

Part 2) run  similar. Again, condition $|\phi(0)|>1$ assert that condition $Y_0$ is satisfied for the full sequence of natural numbers. 

If  $\mathcal{G}'=\{(1/\Psi_n)\}$ is not uniformly bounded on compact subsets then we have two possibilities. a) The family $\{(1/\Psi_n)\}$ is normal at no  point of $\mathbb{D}$. In this case by applying Theorems \ref{birkhoff} and \ref{montel}, there is a complex number $z_1\in \mathbb{D}$ with $\mathcal{G}'$-orbit dense. In particular, we can select a subsequence $\{n_k\}$ such that $(1/\Psi_{n_k})(z_1)\to \infty$. And b), there exists a point $z_1\in \mathbb{D}$ such that
$(1/\Psi_{n_k})(z_1)\to \infty$.

In both cases, we can guarantee that $\Psi_{n_k}(z_1)\to 0$, and as before there exists a dense subsets $X_0$ such that $T^{n_k}x_0\to 0$ for all $x_0\in X_0$. Therefore the hypercyclicity criterion is satisfied, hence $T$ is hypercyclic.
\end{proof}

\begin{remark}
Let us observe that we can improve Theorem \ref{main2}
in such a form: If there exists a subsequence $(n_k)$ such that $\Psi_{n_k}(z_1)\to 0$ for some $z_1\in \mathbb{D}$ and the sequence $\{\Omega_{n_k}\}$ is not uniformly bounded on compact subsets, then $T$ is hypercyclic.

Analogously, we can prove that if for some subsequence $(n_k)$, $\Omega_{n_k}(z_1)\to \infty$ for some $z_1\in\mathbb{D}$ and $1/\Psi_{n_k}$ is not uniformly bounded on compact subsets, then $T$ is hypercyclic.
\end{remark}

\section{Supercyclicity of extended $\lambda$-eigenoperators of the backward shift.}

\label{supercyclicity}

In this section we study if and operator that $\lambda$-commute with the backward shift is supercyclic.
For element of the commutant, this result is true. Moreover, solving a question posed by Godefroy and Shapiro, V. M\"uller was able to prove that any non-scalar operator that commutes with a generalized backward shift is supercyclic (\cite{muller}). Let us recall that
 a bounded linear operator $\mathcal{B}$ on a Banach space
$X$ is a generalized backward shift if it satisfies the following conditions:
\begin{enumerate}
  \item The kernel of $\mathcal{B}$ is one dimensional.
\item $\bigcup\{\ker \mathcal{B}^n:\ n = 0, 1, 2,...\}$ is dense in $X$.
\end{enumerate}
For more information about generalized backward shift see \cite{GoSha}.

It is natural to consider an operator $T$ that $\lambda$-commute with a generalized backward shift $\mathcal{B}$ and to try to see if $T$ is supercyclic. We see that this property is not longer true for all elements of the $\lambda$-commutant of $\mathcal{B}$. Nonetheless the result is true for many elements in the $\lambda$-commutant.

\begin{proposition}
\label{supercyclic1}
Assume that $\mathcal{B}$ is a generalized backward shift on a Banach space $X$, and $A$ is a $\lambda$-extended eigenoperator of $\mathcal{B}$. If $\textrm{ker}(A)\supset \textrm{ker}(\mathcal{B})$ then $A$ is supercyclic.
\end{proposition}
\begin{proof}
Following to Godefroy-Shapiro, if $\mathcal{B}$ is a generalized backward shift, then there exists a sequence of vectors $(x_k)_{n\geq 0}$ such that $\mathcal{B}x_{n}=x_{n-1}$ for $n\geq1$ and $\mathcal{B}x_0=0$.  Moreover, if $A$ commutes with $\mathcal{B}$ then the matrix of $A$ relative to the basis $(x_k)_{k\geq 0}$
is upper triangular and it is constant on each superdiganal. That is, $A$ can be represented as a formal power series of $\mathcal{B}$.

In a similar way, following \cite{Pet}, if $A$ is an  $\lambda$-extended eigenoperator of   $\mathcal{B}$ then $A$ has the following matrix representation with respect to the basis $(x_k)_{k\geq 0}$, namelly, the $(p+1)$ superdiagonal of $A$ has the form $(c_p\lambda^n)_{n\geq 0}$. That is, $Ax_n=c_nx_0+c_{n-1}\lambda x_1+\cdots+c_0\lambda^{n}x_n$.

Thus, if $p$ is the first $p$ for which $c_p\neq 0$ then
$A$ can be expresed as $A=R_{\lambda} B^pA_p$ where $A_p$ is the formal power series $\phi(\mathcal{B})=c_pI+c_{p+1}\mathcal{B}+c_{p+2}\mathcal{B}^2+\cdots$ and $R_\lambda x_n=\lambda^nx_n$. Moreover, if $\textrm{ker}(A)\supset \textrm{ker}(\mathcal{B})$ then $p\geq 1$.

Now, we mimics the proof of Proposition 3.6 in \cite{GoSha}.
If we denote by 
$$Y_{n+1}=\textrm{linear span}\{x_0,x_1,\cdots, x_n\}$$
for $n\geq 0$, then $Y_{n+1}$ is invariant under $A_p$ and $A_p$ is invertible on $Y_{n+1}$. Therefore $A_p$ is invertible on $Y=\textrm{linear span}\{x_n\,:\,n\geq 0\}$.
We consider
$$
C=A_p^{-1} F^p R_{1/\lambda},
$$
where $F$ is the generalized forward shift with respect to the basis $(x_k)_{k\geq 0}$.
Clearly $C$ maps $Y_n$ into $Y_{n+p}$ and $AC=I$ on $Y$.
Let us denote by $\sigma(n)$ the norm of $C$ restricted to $Y_n$, then if $y\in Y_k$ then
\begin{eqnarray*}
\|C^ny\|&=&\|CC^{n-1}y\|\\
&\leq & \sigma(k+(n-1)p) \|C^{n-1}y\|\\
&\leq& \sigma(k+(n-1)p)\sigma(k+(n-2)p)\cdots \sigma(k)\|y\|\\
&\leq & (\sigma(k+(n-1)p))^n \|y\|.
\end{eqnarray*}

By considering $r_n=n (\sigma(n+(n-1)p))^n$
and the sequences of operators $T_n=r_nA^n$ and $S_n=\frac{1}{r_n}C^n$
we see that $T_n$ acting on vectors of $Y_n$ is eventually zero, therefore
$T_n$ converges pointwise to zero on $Y$. On the other hand, for each $y\in Y_k$ we get that $T_nS_n=I$ on $Y$ and
$\|S_ny\|\leq \frac{1}{n}\|y\|$. Thus all requeriments of the hypercyclicity criterion are satisfied for $T_n=r_nA^n$, which implies that $A$ is supercyclic as we wanted to prove.
\end{proof}

\begin{remark}
Suprisingly enough, dropping the hypothesis $\textrm{ker}(A)\supset \textrm{ker}(\mathcal{B})$, we can not assert that an operator $A$ that $\lambda$-commute with a generalized backward shift is supercyclic. In fact we can find examples of extended $\lambda$ eigenoperators of $\mathcal{B}$ which are  not supercyclic. 
 That is, there is not a result analogous  to M\"uller's result for operators in the $\lambda$-commutant of a generalized backward shift. 
For example, following Chan and Shapiro (see \cite{chanshapiro}) we consider a {\it comparison} entire function $\gamma(z)=\sum_{n=0}^\infty \gamma_nz^n$,  $\gamma_n>0$ and $\frac{n\gamma_{n}}{\gamma_{n-1}}$ bounded. On such considerations the differentiation operator $D$ is bounded on the  the Hilbert space $E^2(\gamma)$ of entire functions $\sum a_kz^k$ satisfying
$\sum_{k=0} |a_k|^2 \gamma_n^{-2}<\infty$.  Moreover, clearly $D$ is a generalized backward shift on $E^2(\gamma)$.  If we consider $R_\lambda z^n=\lambda^nz^n$, then  $R_\lambda (I+D)$ with $|\lambda|=1$ and $\lambda$ a root of unity, is hypercyclic on $E^{\gamma}$. However if $|\lambda|<1$ by the results in \cite{GSR} we can obtain that $R_\lambda (I+D)$ is not supercyclic on $E^2(\gamma)$ for $|\lambda|<1$.

\end{remark}

Now we reduce our study to the supercyclicity of extended $\lambda$-eigenoperators of $B$, defined on the Hardy space $H^2(\mathbb{D})$. If $T$ is an extended eigenoperator of $B$ then
$T=R_\lambda \phi(B)$. Where $\phi(B)$ is a formal power series.

If $\phi(0)=0$ and $|\lambda|< 1$ then $T=R_\lambda\phi(B)$ by applying Proposition \ref{supercyclic1} we get that $T$ is supercyclic.

If $\phi(0)\neq 0$, we consider the operator $R_\lambda \frac{1}{\phi(0)}\phi(B)=R_\lambda \psi(B)$.  If $\lambda$ is a root of the unity, since $\psi(0)=1$, by applying Proposition \ref{sufficient} (4), then $R_{\lambda}\psi(B)$ is hypercyclic, therefore $T$ is supercyclic.

Now, let us see that if $\phi(0)\neq 0$ and $|\lambda|<1$
then $T=R_\lambda \phi(B)$ is not supercyclic on $H^2(\mathbb{D})$.

\begin{lemma}
Assume that $\phi\in H^2(\mathbb{D})$ and $\phi(0)=1$. If $|\lambda|<1$ then the sequence $\Phi_n(z)=\phi(z)\phi(\lambda z)\cdots \phi(\lambda^{n-1}z)\in H^2(\mathbb{D})$ for all $n$ and there exists $h\in H^2(\mathbb{D})\setminus \{0\}$ such that $\|\Phi_n-h\|_2\to 0$.
\end{lemma}
\begin{proof}
Indeed, since $|\lambda|<1$ then
$\phi(\lambda z)\cdots \phi(\lambda^{n-1}z)\in H^\infty(\mathbb{D})$ for all $n\geq 1$
$$
\|\Phi_n(z)\|_2=\|M_{\phi(\lambda z)\cdots \phi(\lambda^{n-1}z)}\phi\|_2 \leq \|\phi(\lambda z)\cdots \phi(\lambda^{n-1}z)\|_\infty \|\phi\|_2.
$$
Hence, $\Phi_n\in H^2(\mathbb{D})$ for all $n\geq 1$. 
Now let us see that the infinite product $\prod_{n\geq 1}\phi(\lambda^{n-1}z)$ is convergent uniformly on compact subsets to some $h\in H(\mathbb{D})$.
Since $\Phi_n(0)=1$ for all $n$, then $h\neq 0$.

Indeed, set $\phi(z)=1+\phi_0$ and $C=\|\phi_0\|_2$:
\begin{eqnarray*}
|\phi( z)\cdots \phi(\lambda^{n-1}z)|
&\leq &\frac{\|\Phi_n\|_2}{\sqrt{1-|z|^2}}\\
&\leq &\frac{\|\phi(\lambda z)\cdots \phi(\lambda^{n-1}z)\|_\infty \|\phi\|_2}{\sqrt{1-|z|^2}}\\
&\leq & \frac{\prod_{n\geq 1} \|\phi(\lambda^{n-1}z)\|_{2}}{{\sqrt{1-|z|^2}}}\\
&\leq &\frac{\prod_{n\geq 1} (1+|\lambda|^{n-1}C)}{{\sqrt{1-|z|^2}}},
\end{eqnarray*}
Since $|\lambda|<1$, $\sum_{n\geq 1}C|\lambda|^{n-1}$ is convergent, which gives that the infinite product  $\Phi_n(z)$ converges uniformly on compact subsets to some $h\in H(\mathbb{D})$.  Finally, a similar computation yields that $\Phi_n(z)\in H^2(\mathbb{D})$ is a Cauchy sequence. Indeed, if $n>m$ then
\begin{eqnarray*}
\|\Phi_n(z)-\Phi_m(z)\|_2 &=& \left\|\prod_{k=1}^{n} \phi(\lambda^{k-1}z)-\prod_{k=1}^{m}\phi(\lambda^{k-1}z)\right\|_2 \\
&\leq & \left\|\prod_{k=2}^{m}\phi(\lambda^{k-1}z)\right\|_\infty \|\phi\|_2 \left\|\prod_{k=m+1}^{n} \phi(\lambda^{k-1}z)-1\right\|_\infty
\end{eqnarray*}
The results follows if we show that $h_m(z)=\prod_{n=m+1}^\infty \phi(\lambda^{k-1}z) $ converges uniformly to $1$. Since $h_m(z)$ is the remainder of a convergent infinite product, we get that $h_{m-1}(z)\to 1$ uniformly on compact subsets of $\mathbb{D}$, hence
$$
\sup_{|z|=|\lambda|} |h_{m-1}(z)|=\sup_{|z|=1} |h_m(z)|\to 1
$$
as we wanted.  Therefore $\|\Phi_n-h\|_2\to 0$ as $n\to \infty$ as we desired.

\end{proof}

\begin{theorem}
Assume that $|\lambda|<1$ and  $T=R_\lambda\phi(B)$  is an extended $\lambda$-eigenoperator of $B$.  If $\phi(0)\neq 0$ then $T$ is not supercyclic.
\end{theorem}
\begin{proof}
Indeed, let us denote by 
 $H(\mathbb{D})$ the space of analytic funcions in the unit disk provided with the uniform convergence on compact subsets of the unit disk.

Since $H^2(\mathbb{D})\hookrightarrow H(\mathbb{D})$ if $T$ is  hypercyclic or supercyclic on $H^2(\mathbb{D})$  then the sequence of operators $T^n: H^2(\mathbb{D})\to H(\mathbb{D})$ is hypercyclic or supercyclic on  $H(\mathbb{D})$. Let us remark that the operator $T=R_\lambda\phi(B)$ could be not continuous on $H(\mathbb{D})$, but we don't need such requeriments.

We will see that the projective orbit $\{\lambda T^n(f)\,:\,\lambda\in \mathbb{D}, n\in\mathbb{N}\}$ of some vector $f\in H^2$ is not dense in $H(\mathbb{D})$ endowed with the compact open topology.

Without loss of generality we can suppose that $a_0=\phi(0)=1$. Let turn our attention to the following function:
$$
f_0(z)=\frac{1}{1-z}=\sum_{k=0}^\infty z^k \in H(\mathbb{D}).
$$
Clearly $f_0\in H(\mathbb{D})$ and $f_0$ is an eigenfunction of $B$ associated to the eigenvalue $1$, that is, $Bf_0=f_0$.

By contradiction, suppose that there exists a sequence of complex numbers $(\lambda_n)$ and a function  $f\in H^2$ such that the sequence of functions
$(\lambda_n T^nf)$ is dense in  $H(\mathbb{D})$ endowed with the compact-open topology. 
In particular it will exists a sequence $(n_k)$ such that
$\lambda_{n_k}T^{n_k}f\to f_0$ uniformly on compact subsets of   $\mathbb{D}$.

Since  $B$ is  continuos on $H(\mathbb{D})$ and $Bf_0=f_0$ then
$$
B^m(\lambda_{n_k}T^{n_k}f)\to f_0=B^mf_0
$$
uniformly on compact subsets of $\mathbb{D}$. Hence, the quotient:
\begin{equation}
\label{a1}
\frac{\lambda_{n_k}BT^{n_k}f(z)}{\lambda_{n_k}(T^{n_k}f)(0)}=\frac{BT^{n_k}f(z)}{(T^{n_k}f)(0)}\to f_0(z)
\end{equation}
uniformly on compact subsets of $\mathbb{D}$.

We will show that there exists $m\in \mathbb{N}$
such that for any $z\in \mathbb{D}$ the quotients:
$$
\frac{B^mT^{n}f(z)}{(T^{n}f)(0)}\to 0
$$
as $n\to\infty$, obtaining a contradiction.

Indeed, since $B^mT=\lambda^m TB^m$ we get
$$
B^mT^nf=\lambda^{nm}T^nB^mf.
$$
By hypothesis  $f\in H^2(\mathbb{D})$ and $T^nB^m$ is bounded on $H^2(\mathbb{D})$, therefore 
$T^nB^mf\in H^2(\mathbb{D})$. We can suppose without loss that $\|f\|=1$. Hence, since $\|B\|=1$ then  
\begin{eqnarray}
|\lambda^{nm}T^nB^mf(z)|&\leq &\frac{\|T^nB^mf\|}{\sqrt{1-|z|^2}} \\
\label{a2}
&\leq & \frac{|\lambda|^{nm}\|T\|^n}{\sqrt{1-|z|^2}}.
\end{eqnarray}
Here $\|T\|$ denotes the uniform norm of $T$ as linear operator on the space   $H^2(\mathbb{D})$.

Let us obtain inferior estimates of the sequence  $|(T^{n}f)(0)|$. Recall that $T^n=\R_\lambda \phi(B)\cdots R_\lambda\phi(B)=R_{\lambda}^n\phi(B)\cdots \phi(\lambda^{n-1}B)$. On the other hand, since $|\lambda|<1$ the sequence $\Phi_n\to h\neq 0$ on $H^2(\mathbb{D})$. Since the set of supercyclic vectors is dense, we can suppose without loss of generality that $\langle f,h\rangle \neq 0$.

Thus:

\begin{eqnarray*}
|(T^{n}f)(0)|&=&|\langle T^{n}f,1 \rangle|  \\
&=& |\langle f, M_{\overline{\phi}(z)\cdots\overline{\phi}(\lambda^{n-1}z)} R_{\lambda^n}^\star 1\rangle| \\
&=& |\langle f,\overline{\phi}(z)\cdots\overline{\phi}(\lambda^{n-1}z)\rangle|\\
&\to&  \langle f,h \rangle\neq 0
\end{eqnarray*}
as $n\to \infty$. Using this fact, the  equations (\ref{a1}), (\ref{a2}), and by taking $m$ such that $|\lambda|^m<\|T\|$ we get:

\begin{eqnarray*}
\frac{\lambda_{n_k}BT^{n_k}f(z)}{\lambda_{n_k}(T^{n_k}f)(0)}&=&\frac{BT^{n_k}f(z)}{(T^{n_k}f)(0)}\\
&\leq & \frac{\frac{|\lambda|^{nm}\|T\|^n}{\sqrt{1-|z|^2}}}{|(T^{n}f)(0)|}\to 0
\end{eqnarray*}
as $n\to \infty$. Therefore $f$ cannot be supercyclic for $T$ as we desired to prove.
\end{proof}

\section{Concluding remarks}

The results we have obtained here especially Theorem \ref{main} and Theorem \ref{main2}, indicate
there is a {\it dichotomy  theorem} waiting to be proved
that  establishs the elements of the $\lambda$-commutant of $B$ are either hypercyclic or they have orbits quite regulars (that is, either they converge to zero or the orbits of the inverse operator converge to zero).

Our results on the  hypercyclicity problem of the element of the $\lambda$-commutant: $R_\lambda \phi(B)$  is connected with the shape of the spectrum of $\phi(B)$ and it remains a mysterious  a  characterization of hypercyclicity of $R_\lambda\phi(B)$ in terms of the geometry of the spectrum of $\phi(B)$.

In any case we see that it is particularly striking when we are dealing the $\lambda$-commutant  hypercyclicity problem for Banach space operators. We see that the existence of a uniform norm of the operator make difficult the transference of the hypercyclicity. In strong contrast, in Fréchet spaces such transference is easier (see \cite{BGSR,note}).

\noindent
{\bf Conflicts of Interest:} The authors declare no conflict of interest.

\noindent
{\bf Funding: } The second and the third author are supported by Junta de Andalucía, Consejería de Universidad, Investigación e Innovación: ProyExcel\_00780  ”Operator Theory: An interdisciplinary approach.”

\noindent
{\bf Ethical Conduct: } This article is original, it has not been previously published and it has not been simultaneously submitted for evaluation to another journal.
All authors have contributed substantially to the article without omission of any person.

\noindent
{\bf Data Availability Statements:} Not applicable.


\begin{thebibliography}{10}

\bibitem{ansari}
Shamim~I. Ansari.
\newblock Hypercyclic and cyclic vectors.
\newblock {\em J. Funct. Anal.}, 128(2):374--383, 1995.

\bibitem{BGSR}
Ikram Fatima~Zohra Bensaid, Manuel Gonz\'{a}lez, Fernando Le\'{o}n-Saavedra,
  and Mar\'{\i}a~Pilar Romero de~la Rosa.
\newblock Hypercyclicity of operators that {$\lambda$}-commute with the
  differentiation operator on the space of entire functions.
\newblock {\em J. Funct. Anal.}, 282(8):Paper No. 109391, 23, 2022.

\bibitem{herhypope}
Juan B\`es and Alfredo Peris.
\newblock Hereditarily hypercyclic operators.
\newblock {\em J. Funct. Anal.}, 167(1):94--112, 1999.

\bibitem{bousha}
Paul~S. Bourdon and Joel~H. Shapiro.
\newblock Hypercyclic operators that commute with the {B}ergman backward shift.
\newblock {\em Trans. Amer. Math. Soc.}, 352(11):5293--5316, 2000.

\bibitem{chanshapiro}
Kit~C. Chan and Joel~H. Shapiro.
\newblock The cyclic behavior of translation operators on {H}ilbert spaces of
  entire functions.
\newblock {\em Indiana Univ. Math. J.}, 40(4):1421--1449, 1991.

\bibitem{onlamcomope}
John~B. Conway and Gabriel Pr\v{a}jitur\v{a}.
\newblock On {$\lambda$}-commuting operators.
\newblock {\em Studia Math.}, 166(1):1--9, 2005.

\bibitem{GoSha}
Gilles Godefroy and Joel~H. Shapiro.
\newblock Operators with dense, invariant, cyclic vector manifolds.
\newblock {\em J. Funct. Anal.}, 98(2):229--269, 1991.

\bibitem{GSR}
Manuel González, Fernando León-Saavedra, and María Pilar~Romero de~la Rosa.
\newblock Supercyclic properties of extended eigenoperators of the
  differentiation operator on the space of entire functions, 2022.

\bibitem{lincha}
Karl-G. Grosse-Erdmann and Alfredo Peris~Manguillot.
\newblock {\em Linear chaos}.
\newblock Universitext. Springer, London, 2011.

\bibitem{LaLePeZa}
Miguel Lacruz, Fernando Le\'{o}n-Saavedra, Srdjan Petrovic, and Omid Zabeti.
\newblock Extended eigenvalues for {C}es\`aro operators.
\newblock {\em J. Math. Anal. Appl.}, 429(2):623--657, 2015.

\bibitem{note}
Fernando Le\'{o}n-Saavedra and Mar\'{\i}a~Pilar Romero de~la Rosa.
\newblock A note on frequent hypercyclicity of operators that
  {$\lambda$}-commute with the differentiation operator.
\newblock {\em J. Math. Sci. (N.Y.)}, 266(4):615--620, 2022.

\bibitem{muller}
Vladimir M\"{u}ller.
\newblock On the {S}alas theorem and hypercyclicity of {$f(T)$}.
\newblock {\em Integral Equations Operator Theory}, 67(3):439--448, 2010.

\bibitem{Pet}
Srdjan Petrovic.
\newblock Spectral radius algebras, {D}eddens algebras, and weighted shifts.
\newblock {\em Bull. Lond. Math. Soc.}, 43(3):513--522, 2011.

\bibitem{onorbofele}
S.~Rolewicz.
\newblock On orbits of elements.
\newblock {\em Studia Math.}, 32:17--22, 1969.

\end{thebibliography}
\end{document}